\documentclass[12pt,leqno]{article}

\usepackage{amssymb, graphicx, amscd, amsmath, amssymb, amsthm, array, fancyvrb, relsize, paralist}

\newtheorem{thm}{Theorem}[section]
\newtheorem{lem}[thm]{Lemma}

\numberwithin{equation}{section}

\renewcommand\Re{\operatorname{Re}}

\def\F{\mathbb{F}_q}

\def\3F2{{}_3\hspace{-1pt}F_2}
\def\2F1{{}_2\hspace{-1pt}F_1}
\def\h2F1{{}_2\hspace{-1pt}\widehat{F}_1}
\def\f{\F^*}
\def\e{\varepsilon}

\def\lt{\lambda_2}
\def\olt{\overline{\lt}}
\def\lo{\lambda_1}
\def\olo{\overline{\lo}}
\def\no{\nu_1}
\def\ono{\overline{\no}}
\def\nt{\nu_2}

\def\oc4{\overline{\chi_4}}
\def\onu{\overline{\nu}}
\def\ol{\overline{\lambda}}
\def\omu{\overline{\mu}}
\def\oA{\overline{A}}
\def\oa4{\overline{A_4}}
\def\oaf{\overline{A}_4}
\def\oa8{\overline{A}_8}
\def\oB{\overline{B}}
\def\oC{\overline{C}}
\def\oD{\overline{D}}

\def\l({\left(}
\def\r){\right)}
\def\bar{\begin{array}{r|}}
\def\ear{\end{array}}

\def\Re{\mathrm{\ Re \,}}

\title{SOME MIXED CHARACTER SUM IDENTITIES OF KATZ} 
\author{\\ \\ Ron Evans\\
Department of Mathematics\\
University of California at San Diego\\
La Jolla, CA  92093-0112\\
revans@ucsd.edu
}

\medskip

\date{May 2017}

\begin{document}
\maketitle



\begin{abstract}
A conjecture connected with quantum physics led N. Katz to discover
some amazing mixed character sum identities over a field of $q$
elements, where $q$ is a power of a prime $p>3$.
His proof required deep algebro-geometric techniques,
and he expressed interest in finding a more straightforward direct proof.
Such a proof has been given by Evans and Greene in the case 
$q \equiv 3 \pmod 4$, and in this paper we give a proof for
the remaining case $q \equiv 1 \pmod 4$.
Moreover, we show that the identities are valid for all characteristics $p >2$.
\end{abstract}

\newpage

\maketitle

\section{Introduction}
Let $\F$ be a field of $q$ elements, 
where $q$ is a power of an odd prime $p$. 
Throughout this paper, $A$, $B$, $C$, $D$, 
$\chi$, $\lambda$, $\nu$, $\mu$, $\e$, $\phi$, $A_4$, $A_8$
denote complex multiplicative characters on $\f$, 
extended to map 0 to 0. 
Here $\e$ and $\phi$ always denote the trivial and quadratic characters, 
respectively, while $A_4$ denotes a fixed quartic character when
$q \equiv 1 \pmod4$ and $A_8$ denotes a fixed octic character such that
$A_8^2 = A_4$ when $q \equiv 1 \pmod8$. 
Define $\delta(A)$ to be 1 or 0 according as $A$ is trivial or not,
and let $\delta(j,k)$ denote the Kronecker delta.

For $y \in \F$, 
let $\psi(y)$ denote the additive character 
\begin{equation*}
\psi(y) := 
\exp \Bigg( \frac{2 \pi i}{p} \Big( y^p + y^{p^2} + \dots + y^q \Big) \Bigg).
\end{equation*}
Recall the definitions of the Gauss and Jacobi sums
\begin{equation*}
G(A) = \sum_{y \in \F} A(y) \psi(y), \quad
J(A, B) = \sum_{y \in \F} A(y) B(1-y).
\end{equation*}
These sums have the familiar properties 
\[
G(\e) = -1, \quad J(\e,\e) = q-2,
\]
and for nontrivial $A$, 
\[
G(A) G(\oA) = A(-1) q, \quad J(A, \oA) = -A(-1), 
\quad J(\e, A)=-1. 
\]
Gauss and Jacobi sums are related by \cite[p. 59]{BEW}
\begin{equation*}
J(A,B) = \frac{G(A) G(B)}{G(AB)}, \quad \text{if } AB \neq \e
\end{equation*}
and 
\begin{equation*}
J(A,\oC) = \frac{A(-1)G(A) G(\oA C)}{G(C)}=A(-1)J(A,\oA C), 
\quad \text{if } C \neq \e.
\end{equation*}
The Hasse--Davenport product relation \cite[p. 351]{BEW} yields
\begin{equation}\label{eq:1.1}
A(4) G(A) G(A \phi) = G(A^2) G(\phi).
\end{equation}

As in \cite[p. 82]{TAMS}, 
define the hypergeometric $\2F1$ function over $\F$ by
\begin{equation}\label{eq:1.2}
\2F1 \l( \bar A,B \\ C \ear \ x \r)
=\frac{\e (x)}{q}
\sum_{y \in \F} B(y) \oB C(y-1)
\oA(1-xy), \quad x \in \F.
\end{equation}

For $j, k \in \F$ and $a \in \f$, 
Katz \cite[p. 224]{Katz} defined the mixed exponential sum
\begin{equation}\label{eq:1.3}
\begin{split}
P(j,k): &= \delta(j,k) + \phi(-1)\delta(j,-k) + \\
&  \frac{1}{G(\phi)}\sum_{x \in \f}
\phi(a/x - x)\psi(x(j+k)^2 + (a/x)(j-k)^2).
\end{split}
\end{equation}
Note that $P(j,k)=P(k,j)$ and
$P(-j,k) = \phi(-1)P(j,k)$.
Katz proved an equidistribution conjecture of Wootters 
\cite[p. 226]{Katz}, \cite{ASSW}
connected with quantum physics
by constructing explicit character sums $V(j)$ 
 \cite[pp. 226--229]{Katz}) for which
the identities
\begin{equation}\label{eq:1.4}
P(j,k) = V(j)V(k)
\end{equation}
hold for all $j,k \in \F$.
(The $q$-dimensional vector 
$(V(j))_{j \in \F} $
is a minimum uncertainty state, as described by Sussman and
Wootters \cite{SW}.)
Katz's proof \cite[Theorem 10.2]{Katz} of the identities
\eqref{eq:1.4} required the characteristic $p$ to exceed 3,
in order to guarantee that various sheaves of ranks 2, 3, and 4
have geometric and arithmetic monodromy groups
which are SL(2), SO(3), and SO(4), respectively.

As Katz indicated in \cite[p. 223]{Katz}, his proof of
\eqref{eq:1.4} is quite complex, invoking the theory of
Kloosterman sheaves and their rigidity properties, as well as results
of Deligne \cite{Del} and Beilinson, Bernstein, Deligne \cite{BBD}.  
Katz \cite[p. 223]{Katz} wrote,
``It would be interesting to find direct proofs of these identities."

The goal of this paper is to respond to
Katz's challenge by giving a direct proof of \eqref{eq:1.4}
( a ``character sum proof"
not involving algebraic geometry) in the case $q \equiv 1 \pmod4$.
This has the benefit of making
the demonstration of his useful identities 
accessible to a wider audience of mathematicians and physicists.
Another advantage of our proof is that it works for
all odd characteristics $p$, including $p=3$.
As a bonus, we obtain some interesting character sum
evaluations in terms of Gauss sums; see for example
Theorems 2.1, 3.2, and 5.5.

Our method of proof is to show
(see Sections 4 and 6) 
that the Mellin transforms of
both sides of \eqref{eq:1.4} are equal.  A key feature of our proof
is the application in Lemma 5.1 of the following 
hypergeometric $\2F1$ transformation formula over $\F$
for $q \equiv 1 \pmod4$
\cite[Theorem 3]{EG} :
\begin{equation}\label{eq:1.5}
\2F1 \l( \bar D,DA_4 \\ A_4 \ear \ z^4 \r)
=\oD^4(z-1)
\2F1 \l( \bar D,D^2\phi \\ D\phi \ear \ -\left(\frac{z+1}{z-1} \right)^2 \r),
\end{equation}
which holds for every character $D$ on $\F$ and 
every $z \in \f$ with $z \notin \{1, -1\}$.
The proof of \eqref{eq:1.5}
depends on a recently proved finite field analogue
\cite[Theorem 2]{EG}, \cite[Theorem 17]{JLRST}
of a classical quadratic transformation formula of Gauss.

For $q \equiv 1 \pmod4$, Katz's character sums $V(j)$ are defined
in \eqref{eq:1.6}--\eqref{eq:1.7} below.
In the case
$q \equiv 3 \pmod4$,
the sums $V(j)$ have a more complex definition,
in that they are sums over $\mathbb{F}_{q^2}$ \cite[p. 228]{Katz}.
A direct proof of \eqref{eq:1.4} 
for the case $q \equiv 3 \pmod4$
has been given by Evans and Greene \cite{EG2}.

From here on, let $q \equiv 1 \pmod4$, so that 
there exists a primitive fourth root of unity $i \in \F$.
Thus $\phi(-1) = \phi(i^2) =1$ and
$A_4(-4) = A_4((1+i)^4) =1$.

For $a$ as in \eqref{eq:1.3}, define
\[
\tau = -\sqrt{qA_4(-a)},
\]
where the choice of square root is fixed.
For $q \equiv 1 \pmod4$, the sums $V(j)$ are defined as follows:
\begin{equation}\label{eq:1.6}
V(j): = \tau^{-1}\sum_{x \in \f} A_4(x)\psi(x+aj^4/x), \quad j \in \f,
\end{equation}
while for $j=0$,
\begin{equation}\label{eq:1.7}
V(0): = G(A_4)/\tau + \tau /G(A_4).
\end{equation}

\section{Mellin transform of the sums $V(j)$}

For any character $\chi$, define the Mellin transform
\begin{equation}\label{eq:2.1}
S(\chi): = \sum_{j \in \f} \chi(j)V(j).
\end{equation}
The next theorem gives an evaluation of $S(\chi)$ in terms of Gauss sums.

\begin{thm}
If $\chi$ is not a fourth power, then $S(\chi)=0$.  On the other hand,
if $\chi = \nu^4$, then
\begin{equation}\label{eq:2.2}
S(\chi)= 
\tau^{-1}\onu(a)\sum_{m=0}^3 A_4(a)^{1-m}G(\nu A_4^{m-1})G(\nu A_4^{m}).
\end{equation}
\end{thm}

\begin{proof}
By \eqref{eq:1.6}, 
\[
V(j)=V(ji), \quad j \in \f.
\]
Thus $S(\chi)=0$ when $\chi(i) \ne 1$, i.e., when $\chi$ is not a
fourth power.

Now set $\chi = \nu^4$ for some character $\nu$,
and write $\lambda = \nu^2 \oaf$.  By \eqref{eq:1.6}, 
\[
\tau V(j) = \phi(j)\sum_{x \in \f}A_4(x)\psi(j^2(x+a/x)),
\]
so
\begin{equation}\label{eq:2.3}
\begin{split}
\tau S(\chi) &= \sum_{j,x \in \f} \lambda(j^2) A_4(x) \psi(j^2 (x+a/x))\\
&=\sum_{j,x \in \f} \lambda(j)A_4(x)\psi(j(x+a/x))(1+\phi(j)).
\end{split}
\end{equation}

First suppose that $\lambda^2$ is nontrivial.
Then for the sums in \eqref{eq:2.3}, there is no
contribution from the terms where $x+a/x=0$.   Thus
\begin{equation}\label{eq:2.4}
\tau S(\chi) = G(\lambda)Y(\lambda) + G(\lambda \phi)Y(\lambda \phi),
\end{equation}
where
\begin{equation}\label{eq:2.5}
Y(\lambda): = \sum_{x \in \f} A_4(x) \ol(x+a/x).
\end{equation}
By \eqref{eq:2.5},
\begin{equation*}
\begin{split}
Y(\lambda) &= Y(\nu^2 \oaf)=\sum_{x \in \f} A_4(x) \onu^2 A_4(x+a/x)=
\sum_{x \in \f} \nu(x^2) \onu^2 A_4(x^2+a)\\
&= \sum_{x \in \f} \nu(x) \onu^2 A_4(x+a) (1 + \phi(x))\\
&= \onu(-a)\{A_4(a)J(\nu, \onu^2 A_4) + \oaf (a)J(\nu\phi, \onu^2 A_4)\}.
\end{split}
\end{equation*}
Since $J(B,C) = B(-1)J(B,\overline{BC})$ for all characters $B$, $C$,
we see that
\[
J(\nu, \onu^2 A_4)= \nu(-1) J(\nu, \nu \oaf), \quad
J(\nu\phi, \onu^2 A_4)= \nu(-1) J(\nu \phi, \nu A_4).
\]
Thus
\begin{equation}\label{eq:2.6}
Y(\lambda)=\onu(a)\{A_4(a)J(\nu, \nu \oaf) + \oaf (a)J(\nu\phi, \nu A_4)\}.
\end{equation}
Similarly, we have
\begin{equation}\label{eq:2.7}
Y(\lambda \phi)=Y(\nu^2 A_4)=
\onu(a)\{J(\nu, \nu A_4) + \phi(a)J(\nu\phi, \nu \oaf)\}.
\end{equation}
Putting \eqref{eq:2.6}--\eqref{eq:2.7} into \eqref{eq:2.4},
we easily see that
\eqref{eq:2.2} holds in the case that $\lambda^2$ is nontrivial.

Finally, assume that $\lambda^2$ is trivial, so that 
$\nu^4 = \chi=\phi$.  Then
$q \equiv 1 \pmod 8$ and $\nu$ is an odd power of $A_8$.
By \eqref{eq:2.3},
\begin{equation}\label{eq:2.8}
\begin{split}
\tau S(\chi) &= \sum_{j \in \F}\sum_{x \in \f} A_4(x) \psi(j^2 (x+a/x))\\
&= q \sum_{\substack{x \in \f  \\ x+a/x = 0}}A_4(x) \
+  \ G(\phi)\sum_{x \in \f} A_4(x) \phi(x+a/x).
\end{split}
\end{equation}
The first of the two terms on the far right of \eqref{eq:2.8}
vanishes when $\phi(-a)=-1$,
and so this term equals $q(A_8(-a) + A_8^5(-a))$.  
Thus \eqref{eq:2.8} yields
\begin{equation*}
\begin{split}
\tau S(\chi) &= q(A_8(-a) + A_8^5(-a))+
G(\phi) \sum_{x \in \F} \oa8 (x)\phi(a+x)(1+\phi(x))\\
&= q(A_8(-a) + A_8^5(-a))+
G(\phi) \{A_8^3(-a)J(\oa8, \phi) + \oa8(-a) J(A_8^3, \phi)\}.
\end{split}
\end{equation*}
It follows that
\begin{equation}\label{eq:2.9}
\begin{split}
\tau S(\chi) &= A_8(a)G(A_8)G(\oa8) + A_8^5(a)G(A_8^3)G(\oa8^3)\\
&+ A_8^3(a)G(\oa8)G(\oa8^3) +\oa8 (a)G(A_8)G(A_8^3).
\end{split}
\end{equation}
No matter which odd power of $A_8$ is substituted for $\nu$
in \eqref{eq:2.2}, we see that \eqref{eq:2.2} matches \eqref{eq:2.9}.
Thus the proof of \eqref{eq:2.2} is complete.
\end{proof}

\section{Mellin transform of the sums $P(j,0)$}

For any character $\chi$, define the Mellin transform
\begin{equation}\label{eq:3.1}
T(\chi): = \sum_{j \in \f} \chi(j)P(j,0).
\end{equation}
Theorem 3.2 below gives an evaluation of $T(\chi)$ in terms of Gauss sums.
We will need the following lemma.

\begin{lem}
When $\nu^4$ is nontrivial,
\begin{equation*}
\2F1 \l( \bar \nu^2,\nu A_4 \\ \nu \oaf \ear \ -1 \r) =
\frac{A_4(-1)G(\nu A_4)}{qG(\phi)G(\nu^2)}\{G(\nu)G(A_4)+G(\nu \phi)G(\oaf)\}.
\end{equation*}
\end{lem}

\begin{proof}
This follows by first permuting the numerator parameters by means of
\cite[Corollary 3.21]{TAMS}
and then applying a finite field analogue \cite[(4.11)]{TAMS}
of a classical summation formula of Kummer.
\end{proof}

\begin{thm}
If $\chi$ is not a fourth power, then $T(\chi)=0$.  On the other hand,
if $\chi = \nu^4$, then
\begin{equation}\label{eq:3.2}
\begin{split}
T(\chi)&= q^{-1}A_4(-1)\{\oaf(a)G(A_4) + G(\oaf)\} \ \times \\
&\times \onu(a) \sum_{m=0}^3 A_4(a)^{1-m}G(\nu A_4^m)G(\nu A_4^{m-1}).
\end{split}
\end{equation}
\end{thm}

\begin{proof}
By \eqref{eq:3.1} and \eqref{eq:1.3},
\[
G(\phi)T(\chi) = \sum_{j,x \in \f}
\phi(x - a/x) \chi(j) \psi(j^2 (x +a/x)).
\]
Therefore $T(\chi)=0$ unless $\chi$ is a square, so suppose that
$\chi = (\lambda A_4)^2 $
for some character $\lambda$.  Then
\begin{equation}\label{eq:3.3}
\begin{split}
G(\phi)T(\chi) &= \sum_{j,x \in \f} 
\phi(x - a/x) \lambda A_4(j)\psi(j(x +a/x)) (1 + \phi(j)) \\
&= 
U(\lambda) +W(\lambda) +W(\lambda \phi),
\end{split}
\end{equation}
where
\begin{equation}\label{eq:3.4}
U(\lambda) = \sum_{\substack{x \in \f  \\ x+a/x = 0}}
\phi(x - a/x)\sum_{j \in \f}
(\lambda A_4(j) + \lambda \oaf(j))
\end{equation}
and
\begin{equation}\label{eq:3.5}
W(\lambda)= G(\lambda A_4)\sum_{x \in \f}
\phi(x - a/x) \ol \oaf (x+a/x).
\end{equation}
Thus $T(\chi)=0$ unless $\lambda A_4=\nu^2$ for some
character $\nu$, i.e., unless $\chi = \nu^4$.   This proves the
first part of Theorem 3.2.
For the remainder of this proof, assume that 
\[
\lambda A_4 = \nu^2, \quad \chi = \nu^4.
\]

First consider the case where $\chi = \nu^4$ is nontrivial.
Then $U(\lambda)=0$, since $\lambda A_4$ 
and $\lambda \oaf$ are nontrivial.

We have
\begin{equation}\label{3.6}
\begin{split}
&W(\lambda)/G(\nu^2) = 
\sum_{x \in \f} \phi(x - a/x) \onu^2(x + a/x)\\
&= \sum_{x \in \f}\nu^2 \phi(x) \phi(a - x^2) \onu^2 (a + x^2)\\
&= \sum_{x \in \f} \nu A_4(x)\phi(a-x)\onu^2(a+x)(1 + \phi(x)) \\
&= q \onu \oaf (a)  
\2F1 \l( \bar \nu^2,\nu A_4 \\ \nu \oaf \ear \ -1 \r) +
q\onu A_4 (a) 
\2F1 \l( \bar \nu^2,\nu \oaf \\ \nu A_4 \ear \ -1 \r).
\end{split}
\end{equation}
By Lemma 3.1,
\begin{equation}\label{eq:3.7}
\begin{split}
&W(\lambda) = \\
&q^{-1}A_4(-1)G(\nu A_4)G(\phi) \onu \oaf (a)
\{G(\nu)G(A_4)+G(\nu \phi)G(\oaf)\} \\
+ \ & q^{-1}A_4(-1)G(\nu \oaf)G(\phi) \onu A_4 (a)
\{G(\nu \phi)G(A_4)+G(\nu)G(\oaf)\} .
\end{split}
\end{equation}
Similarly,
\begin{equation}\label{eq:3.8}
\begin{split}
&W(\lambda \phi) =\\ 
&q^{-1}A_4(-1)G(\nu \phi)G(\phi) \onu \phi(a)
\{G(\nu A_4)G(A_4)+G(\nu \oaf)G(\oaf)\} \\
+ \ & q^{-1}A_4(-1)G(\nu )G(\phi) \onu (a)
\{G(\nu \oaf)G(A_4)+G(\nu A_4 )G(\oaf)\} .
\end{split}
\end{equation}
By \eqref{eq:3.3}, \eqref{eq:3.7}, and \eqref{eq:3.8},
we arrive at the desired result \eqref{eq:3.2}
in the case where $\chi = \nu^4$ is nontrivial.

Finally, suppose that $\chi=\nu^4$ is trivial, 
so that either $\lambda A_4 =\e$
or $\lambda A_4 =\phi$.
We must show that
\begin{equation}\label{eq:3.9}
\begin{split}
\sum_{j \in \f}P(j,0) &= q^{-1}A_4(-1)\{\oaf(a)G(A_4) + G(\oaf)\} \ \times \\
&\times \sum_{m=0}^3 A_4(a)^{1-m}G(A_4^m)G(A_4^{m-1}).
\end{split}
\end{equation}
Straightforward computations show that
\begin{equation}\label{eq:3.10}
U(\lambda) = (q-1)(A_4(a) + \oaf(a))
\end{equation}
and
\begin{equation}\label{eq:3.11}
\begin{split}
W(\lambda) + W(\lambda \phi)=& 
A_4(a) + \oaf(a) - A_4(a)J(\oaf, \phi) - \oaf(a)J(A_4, \phi)\\
& -2G(\phi) +\phi(a)J(A_4,\phi)G(\phi) + \phi(a) J(\oaf, \phi)G(\phi).
\end{split}
\end{equation}
The desired result \eqref{eq:3.9} now follows from
\eqref{eq:3.3}, \eqref{eq:3.10}, and \eqref{eq:3.11}.
\end{proof}

\section{Proof of  (1.4) when $jk=0$}

We first consider the case where $j=k=0$.  By \eqref{eq:1.3},
\begin{equation*}
\begin{split}
G(\phi)(P(0,0)-2) & =   \sum_{x \in \f}\phi(x)\phi(x^2 -a)=
\sum_{u \in \f}A_4(u)\phi(u-a)(1+\phi(u))\\
& = 2\Re{\oaf (a) J(A_4,\phi)} = 2\Re{A_4(4/a)J(A_4,A_4)},
\end{split}
\end{equation*}
where the last equality follows from the Hasse--Davenport formula (1.1).
Dividing by $G(\phi)$ and using the fact that $A_4(4)=A_4(-1)$, we have
\[
P(0,0)=2 + 2\Re{G(A_4)^2/(qA_4(-a))}.
\]
It now follows easily from \eqref{eq:1.7} that $P(0,0) = V(0)^2$.

To complete the proof of \eqref{eq:1.4} for $jk=0$, it 
remains to prove that 
\begin{equation}\label{eq:4.1}
P(j,0) = V(0)V(j), \quad j \in \f,
\end{equation}
since $P$ is symmetric in its two arguments.
By Theorems 2.1 and 3.2, the Mellin transforms of the left
and right sides of \eqref{eq:4.1} are the same for all characters.
By taking inverse Mellin transforms, we see that \eqref{eq:4.1}
holds, so the proof of \eqref{eq:1.4} for $jk=0$ is complete.

\section{Double Mellin Transform  of $P(j,k)$}

For characters $\chi_1, \chi_2$, define the double Mellin transform
\begin{equation}\label{eq:5.1}
T=T(\chi_1, \chi_2): = \sum_{j, k \in \f} \chi_1(j)\chi_2(k)P(j,k).
\end{equation}
Note that $T(\chi_1, \chi_2)$ is symmetric in $\chi_1$, $\chi_2$.
In this section we will evaluate $T$.
Theorem 5.4 shows that $T=0$ when $\chi_1$ and $\chi_2$
are not both fourth powers.  Theorem 5.5 evaluates $T$
when $\chi_1$ and $\chi_2$  are both fourth powers.

Since $P(j,k)=P(-j,k)$, we have $T=0$ unless  
$\chi_1$ and $\chi_2$ are squares, so we set
\begin{equation}\label{eq:5.2}
\chi_i = (\lambda_i A_4)^2 = \lambda_i^2 \phi, \quad i=1,2
\end{equation}
for characters $\lambda_i$ (which are 
well-defined up to factors of $\phi$).

From the definitions of $T$ and $P(j,k)$, we have
\begin{equation*}
\begin{split}
&G(\phi)\{T - (2q-2)\delta(\lo^2\lt^2)\}= \\
&\sum_{j,k,x \in \f} \lo^2\phi(j) \lt^2\phi(k)\phi(x - a/x)
\psi(x(j+k)^2 +a(j-k)^2/x).
\end{split}
\end{equation*}

Replace $j$ by $jk$ to obtain
\begin{equation}\label{eq:5.3}
\begin{split}
&G(\phi)\{T - (2q-2)\delta(\lo^2\lt^2)\}= \\
&\sum_{j,k,x \in \f} \lo^2\phi(j) \lo^2\lt^2(k) \phi(x - a/x)
\psi(k^2\alpha(j,x)) = \\
&\sum_{j,k,x \in \f} \lo^2\phi(j) \lo\lt(k) \phi(x - a/x)
\psi(k\alpha(j,x))(1+\phi(k)),
\end{split}
\end{equation}
where
\begin{equation}\label{eq:5.4}
\alpha(j,x) := x(j+1)^2 +a(j-1)^2/x.
\end{equation}
Note that $\alpha(j,x)$ cannot vanish when $j = \pm 1$.
By \eqref{eq:5.3}.
\begin{equation}\label{eq:5.5}
\begin{split}
&G(\phi)\{T - (2q-2)\delta(\lo^2\lt^2)\}= \\
&\delta(\lo^2\lt^2)(q-1)H(\lo)
+G(\lo\lt)E(\lo, \lt)+G(\lo\lt\phi)E(\lo, \lt\phi),
\end{split}
\end{equation}
where
\begin{equation}\label{eq:5.6}
H(\lo):= \sum_{\substack{j,x \in \f \\ \alpha(j,x)=0}}
\lo^2\phi(j)\phi(x-a/x)
\end{equation}
and
\begin{equation}\label{eq:5.7}
E(\lo, \lt): = \sum_{j,x \in \f}
\lo^2\phi(j)\phi(x-a/x)\olo\olt(\alpha(j,x)).
\end{equation}

For a character $D$
and $j \in \f$, define
\begin{equation}\label{eq:5.8}
h(D, j):=\sum_{x \in \f}D(x)\phi(1-x)
\oD^2\phi(x(j+1)^2+(j-1)^2).
\end{equation}

By \eqref{eq:5.7},
\begin{equation}\label{eq:5.9}
E(\lo, \lt) = \sum_{j \in \f} \chi_1(j)\beta(\lo,\lt,j),
\end{equation}
where
\begin{equation}\label{eq:5.10}
\beta(\lo,\lt,j):=\sum_{x \in \f} \phi(a-x^2)\lo\lt\phi(x)
\olo\olt(x^2(j+1)^2 + a(j-1)^2).
\end{equation}
If $\lo\lt\phi$ is odd, 
i.e., if $\chi_1 \chi_2=\lo^2 \lt^2$ is not a fourth power, then
we'd have
$\delta(\lo^2 \lt^2)=0$ and
\begin{equation}\label{eq:5.11}
\beta(\lo,\lt,j)=\beta(\lo,\lt \phi,j)=0,
\end{equation}
so that $T=0$ by \eqref{eq:5.5}.
Thus we may assume that
\begin{equation}\label{eq:5.12}
\lo\lt\phi= \mu^2
\end{equation}
for some character $\mu$ which is well defined up to factors of $A_4$. 
By \eqref{eq:5.10},
\begin{equation}\label{eq:5.13}
\begin{split}
\beta(\lo,\lt,j)&=\sum_{x \in \f}\phi(a-x)\omu^2\phi(x(j+1)^2+a(j-1)^2)
\{\mu(x) +  \mu\phi(x)\}\\
&= \omu(a)h(\mu, j) + \omu\phi(a)h(\mu\phi, j)
\end{split}
\end{equation}
and 
(by replacing $\lt$ by $\lt\phi$)
\begin{equation}\label{eq:5.14}
\begin{split}
\beta(\lo,\lt \phi,j)&=\sum_{x \in \f}\phi(a-x)\omu^2(x(j+1)^2+a(j-1)^2)
\{\mu A_4(x) +  \mu \oaf(x)\}\\
&= \omu \oaf(a)h(\mu A_4, j) + \omu A_4(a)h(\mu \oaf, j).
\end{split}
\end{equation}
Thus  \eqref{eq:5.5} is equivalent to
\begin{equation}\label{eq:5.15}
\begin{split}
&G(\phi)\{T - (2q-2)\delta(\mu^4)\}= 
\delta(\mu^4)(q-1)H(\lo)\\
&+\omu(a)\sum_{m=0}^3 G(\mu^2\phi^{m+1})\oaf^m(a)
\sum_{j \in \f} \chi_1(j)h(\mu A_4^m, j).
\end{split}
\end{equation}

For $D \in \{\e, A_4, \oaf \}$ and $j \in \f$,
we can evaluate $h(D,j)$ directly
from definition \eqref{eq:5.8}, as follows:

\begin{equation}\label{eq:5.16}
h(\e, j) = -2 +q\delta(j^2,-1) + \delta(j^2,1),
\end{equation}

\begin{equation}\label{eq:5.17}
h(A_4, j) = J(A_4, \phi) - \phi(j^4 -1),
\end{equation}

\begin{equation}\label{eq:5.18}
h(\oaf, j) = J(\oaf, \phi) - \phi(j^4 -1).
\end{equation}

For $D \notin \{\e, A_4, \oaf \}$,
the following lemma expresses $h(D,j)$ in terms of a 
hypergeometric character sum.  

\begin{lem}
For $D \notin \{\e, A_4, \oaf \}$ and $j \in \f$,
\begin{equation}\label{eq:5.19}
h(D,j) = \frac{G(D)^2 G(\phi)}{G(D^2\phi)}
\2F1 \l( \bar D, DA_4 \\ A_4 \ear \ j^4 \r).
\end{equation}
\end{lem}

\begin{proof}
First consider the case where $j \ne \pm1$.  By \eqref{eq:1.2},
\begin{equation}\label{eq:5.20}
h(D,j) = q\oD^4(j-1)
\2F1 \l( \bar D^2\phi,D \\ D\phi \ear \
-\left(\frac{j+1}{j-1}\right)^2  \r).
\end{equation}
Since $D \notin \{\e, A_4, \oaf \}$, we can permute the
two numerator parameters by means of
\cite[Corollary 3.21]{TAMS} to obtain
\begin{equation}\label{eq:5.21}
h(D,j):=\frac{G(D)^2 G(\phi)}{G(D^2\phi)}
\oD^4(j-1)\2F1 \l( \bar D,D^2\phi \\ D\phi \ear \
-\left(\frac{j+1}{j-1}\right)^2  \r).
\end{equation}
Now \eqref{eq:5.19} for $j \ne \pm1$ follows from \eqref{eq:1.5}.

Finally, let $j = \pm 1$. By a finite field analogue \cite[Theorem 4.9]{TAMS}
of Gauss's classical summation formula,
\begin{equation}\label{eq:5.22}
\2F1 \l( \bar D, DA_4 \\ A_4 \ear \ j^4 \r)=
\frac{D(-1)G(DA_4)G(\oD^2)}{qG(\oD A_4)}.
\end{equation}
Thus by the Hasse-Davenport formula \eqref{eq:1.1} with $A=\oD A_4$,
\begin{equation}\label{eq:5.23}
\2F1 \l( \bar D, DA_4 \\ A_4 \ear \ j^4 \r)=
\frac{\oD(4)G(\oD^2)}{G(\oD^2 \phi)G(\phi)}.
\end{equation}
From \eqref{eq:5.8} with $j = \pm1$,
\[
h(D,j) = \oD(4)^2 J(D, \phi) = \oD(4)G(D)^2/G(D^2),
\]
where the last equality follows from 
\eqref{eq:1.1}.  Together with \eqref{eq:5.23}, 
this completes the proof of \eqref{eq:5.19}
for $j = \pm 1$.
\end{proof}

\begin{lem}
If $\chi_1$ is not a fourth power, then $H(\lo)=0$.
On the other hand, if $\chi_1 = \no^4$, then
\begin{equation}\label{eq:5.24}
H(\lo)= (A_4(a)+\oaf(a))\sum_{m=0}^3 J(\no A_4^m, \phi).
\end{equation}
\end{lem}

\begin{proof}
We have
\begin{equation*}
\begin{split}
H(\lo)&=(A_4(-a)+\oaf(-a)) \sum_{j \ne \pm 1} 
\chi_1(j)\phi(\frac{j-1}{j+1}+\frac{j+1}{j-1}) \\
&=(A_4(-a)+\oaf(-a)) \sum_{j \in \f} \chi_1(j)\phi(2)
\phi(j^2+1)\phi(j^2-1) \\
&=(A_4(a)+\oaf(a)) \sum_{j \in \f} \chi_1(j)\phi(j^4 -1).
\end{split}
\end{equation*} 
If $\chi_1$ is not a fourth power, then replacement of $j$ by $ji$  shows that
$H(\lo)=0$.  On the other hand, 
if $\chi_1 = \no^4$, then
\begin{equation*}
H(\lo)=(A_4(a)+\oaf(a)) \sum_{j \in \f}\no(j)\phi(1-j)
\sum_{m=0}^3 A_4^m(j),
\end{equation*}
which proves \eqref{eq:5.24}
\end{proof}

\begin{lem}
If $\mu^4$ is 
nontrivial and $\chi_1$ is not a fourth power, then $T=0$.
\end{lem}

\begin{proof}
Suppose that $\mu$ is not a power of $A_4$. It follows from Lemma 5.1
that each expression $h(\mu A_4^m, j)$ in \eqref{eq:5.15}
is unchanged when $j$ is replaced by $ji$.
If $\chi_1$ is not a fourth power, then
each sum on $j$ in \eqref{eq:5.15} vanishes,
and hence $T$ vanishes.
\end{proof}

\begin{thm}
$T(\chi_1, \chi_2)=0$ when
the characters $\chi_1$ and $\chi_2$ are not both fourth powers.
\end{thm}

\begin{proof}
Since $T(\chi_1, \chi_2)$ is symmetric in 
the arguments $\chi_1$, $\chi_2$,
it suffices to prove that $T=0$ under the
assumption that $\chi_1$ is not a fourth power.
In view of Lemma 5.3, we may also assume that
$\mu^4$ is trivial, i.e., $\mu A_4^n$ is trivial
for some $n \in \{0,1,2,3\}$.
By \eqref{eq:5.16}--\eqref{eq:5.19},
$h(\mu A_4^m, j)$ is unchanged when $j$
is replaced by $ji$, unless $m=n$.
Since $H(\lo)=0$ by Lemma 5.2, it follows from \eqref{eq:5.15}
and \eqref{eq:5.16} that
\begin{equation}\label{eq:5.25}
G(\phi)T = G(\phi)(2q-2) + G(\phi)
\sum_{j \in \f}\chi_1(j)\{-2 + q\delta(j^2,-1)+\delta(j^2,1)\}.
\end{equation}
Thus
\begin{equation}\label{eq:5.26}
T = (2q-2) + q\chi_1(i) + q\chi_1(-i)+\chi_1(1)+\chi_1(-1).
\end{equation}
Since $\chi_1$ is a square by \eqref{eq:5.2},
\[
\chi_1(i) = \chi_1(-i)= -1, \quad \chi_1(-1) = \chi_1(1)= 1,
\]
so that \eqref{eq:5.26} yields the desired result $T=0$.
\end{proof}

The next theorem evaluates
$T(\chi_1, \chi_2)$ when
$\chi_1$ and $\chi_2$ are both fourth powers.
By \eqref{eq:5.2}, we may assume that
\begin{equation}\label{eq:5.27}
\lambda_i A_4 = \nu_i^2, \ \chi_i = \nu_i^4, \quad i=1,2,
\end{equation}
for characters $\no$ and $\nt$
(which are well-defined up to factors of $A_4$).  By \eqref{eq:5.12},
$(\no\nt)^2 =\mu^2$, and we may assume that
\begin{equation}\label{eq:5.28}
\mu = \no\nt,
\end{equation}
otherwise replace each $\nu_i$ with $\nu_i A_4$.

\begin{thm}
Suppose that
$\chi_i = \nu_i^4$ for $i=1,2$. Then
\begin{equation}\label{eq:5.29}
T=\frac{A_4(-a)}{q}\sum_{m=0}^3
\sum_{n=0}^3 \omu\oaf^{m+n}(a) G(\no A_4^{n-1})
G(\no A_4^n)G(\nt A_4^{m-1})G(\nt A_4^m).
\end{equation}
\end{thm}

\begin{proof}
First assume that $\mu^4$ is nontrivial.
By \eqref{eq:5.15},
\begin{equation}\label{eq:5.30}
G(\phi)T=
\sum_{m=0}^3 G(\mu^2\phi^{m+1})\omu\oaf^m(a)
\sum_{j \in \f} \chi_1(j)h(\mu A_4^m, j).
\end{equation}
By Lemma 5.1 and \eqref{eq:1.2},
\begin{equation}\label{eq:5.31}
h(\mu A_4^m, j)=\frac{G(\mu A_4^m)^2 G(\phi)}{qG(\mu^2 \phi^{m+1})}
\sum_{x \in \f}\mu A_4^{m+1}(x) \omu \oaf^m(x-1)\omu \oaf^m(1-xj^4).
\end{equation}
Thus
\begin{equation}\label{eq:5.32}
\begin{split}
&\sum_{j \in \f} \chi_1(j)h(\mu A_4^m, j)=
\frac{G(\mu A_4^m)^2 G(\phi)}{qG(\mu^2 \phi^{m+1})} \times \\
\times &\sum_{x \in \f}\mu A_4^{m+1}(x) \omu \oaf^m(x-1)
\sum_{n=0}^3\sum_{j \in \f} \no A_4^n(j) \omu\oaf^m(1-xj).
\end{split} 
\end{equation}
Replacing $j$ by $j/x$ on the right side of \eqref{eq:5.32},
we see from \eqref{eq:5.30} that
\begin{equation}\label{eq:5.33}
\begin{split}
T&=\sum_{m=0}^3\sum_{n=0}^3\frac{\omu\oaf^m(a)G(\mu A_4^m)^2}{q}
\sum_{x \in \f}\mu\ono A_4^{m-n+1}(x)\omu \oaf^m(x-1)
J(\no A_4^n, \omu\oaf^m)\\
&=\sum_{m=0}^3\sum_{n=0}^3\frac{\omu\oaf^m(a)G(\mu A_4^m)^2}{q}
\mu A_4^m(-1)J(\mu\ono A_4^{m-n+1}, \omu\oaf^m)
J(\no A_4^n, \omu\oaf^m).
\end{split}
\end{equation}
Since $\mu A_4^m$ is nontrivial for each $m$, \eqref{eq:5.33} yields
\begin{equation}\label{eq:5.34}
T=\sum_{m=0}^3\sum_{n=0}^3\frac{\omu\oaf^m(a)A_4(-1)}{q}
G(\no A_4^{n-1})G(\nt A_4^{m-n})G(\nt A_4^{m+1-n})G(\no A_4^{n}),
\end{equation}
in view of the first equality above \eqref{eq:1.1}.
Replacing $m$ by $m+n-1$ in \eqref{eq:5.34}, we complete the proof
of \eqref{eq:5.29} in the case that $\mu^4$ is nontrivial.

Next suppose that $\mu^4$ is trivial, i.e., $\mu$ is a power of $A_4$.
By \eqref{eq:5.15} and Lemma 5.2,
\begin{equation}\label{eq:5.35}
\begin{split}
&G(\phi)T = G(\phi)(2q-2) +(q-1)(A_4(a) + 
\oaf(a))\sum_{m=0}^3 J(\no A_4^m, \phi)\\
&-\oaf(a)\sum_{j \in \f}\chi_1(j)h(A_4,j)
-A_4(a)\sum_{j \in \f}\chi_1(j)h(\oaf,j)\\
&+G(\phi)\sum_{j \in \f}\chi_1(j)h(\e,j)
+\phi(a)G(\phi)\sum_{j \in \f}\chi_1(j)h(\phi,j).
\end{split}
\end{equation}
The formula \eqref{eq:5.35} can be rewritten as
\begin{equation}\label{eq:5.36}
T = \sum_{k=0}^3 R_k A_4^k(a),
\end{equation}
where
\begin{equation}\label{eq:5.37}
R_0 = (2q-2)+\sum_{j \in \f}\chi_1(j)h(\e,j),
\end{equation}
\begin{equation}\label{eq:5.38}
R_1 = G(\phi)^{-1}\{(q-1)\sum_{m=0}^3 J(\no A_4^m, \phi)-
\sum_{j \in \f}\chi_1(j)h(\oaf,j)\},
\end{equation}
\begin{equation}\label{eq:5.39}
R_3 = G(\phi)^{-1}\{(q-1)\sum_{m=0}^3 J(\no A_4^m, \phi)-
\sum_{j \in \f}\chi_1(j)h(A_4,j)\},
\end{equation}
\begin{equation}\label{eq:5.40}
R_2 = \sum_{j \in \f}\chi_1(j)h(\phi,j).
\end{equation}
By \eqref{eq:5.16},
\begin{equation}\label{eq:5.41}
R_0 = 4q -(2q-2)\delta(\chi_1).
\end{equation}
By \eqref{eq:5.18},
\begin{equation}\label{eq:5.42}
R_1 = G(\phi)^{-1}
\{q\sum_{m=0}^3 J(\no A_4^m, \phi) -\delta(\chi_1)(q-1)J(\oaf, \phi)\}.
\end{equation}
By \eqref{eq:5.17},
\begin{equation}\label{eq:5.43}
R_3 = G(\phi)^{-1}
\{q\sum_{m=0}^3 J(\no A_4^m, \phi) -\delta(\chi_1)(q-1)J(A_4, \phi)\}.
\end{equation}
By Lemma 5.1 and \eqref{eq:1.2},
\begin{equation}\label{eq:5.44}
R_2 = \sum_{m=0}^3 J(\ono \oaf^{m+1}, \phi)J(\no A_4^{m}, \phi).
\end{equation}
A lengthy but straightforward computation now shows that
for each $k$ in $\{0,1,2,3\}$,
\begin{equation}\label{eq:5.45}
R_k = q^{-1} A_4(-1)\sum G(\no A_4^{n-1})G(\no A_4^{n})
G(\ono A_4^{m-1})G(\ono A_4^{m}),
\end{equation}
where the sum is over all $m, n \in \{0,1,2,3\}$ for which
$\oaf^{m+n-1} = A_4^k$. 
Putting \eqref{eq:5.45} into \eqref{eq:5.36}, we obtain
\begin{equation}\label{eq:5.46}
T=\frac{A_4(-a)}{q}\sum_{m=0}^3
\sum_{n=0}^3 \oaf^{m+n}(a) G(\no A_4^{n-1})
G(\no A_4^n)G(\ono A_4^{m-1})G(\ono A_4^m).
\end{equation}
Since $\omu = A_4^\ell$ for some $\ell$, we may substitute
$\nt A_4^\ell$ for $\ono$ in \eqref{eq:5.46}.  Then upon replacing
$m$ by $m-\ell$, we complete the proof of \eqref{eq:5.29} in the case
that  $\mu^4$ is trivial.
\end{proof}

\section{Proof of Katz's identities (1.4)}

The proof for $jk=0$ was given in Section 4, so we may assume
that $jk \ne 0$.
Let 
\begin{equation}\label{eq:6.1}
S(\chi_1, \chi_2):= \sum_{j \in \f}\sum_{k \in \f}
\chi_1(j)\chi_2(k)V(j)V(k)
\end{equation}
denote the double Mellin transform of $V(j)V(k)$.
In the notation of \eqref{eq:2.1},
\begin{equation}\label{eq:6.2}
S(\chi_1, \chi_2)= S(\chi_1)S(\chi_2).
\end{equation}
If $\chi_1$ and $\chi_2$ are not both fourth powers, then
\begin{equation}\label{eq:6.3}
S(\chi_1, \chi_2)=T(\chi_1, \chi_2),
\end{equation}
since both members of \eqref{eq:6.3} vanish by Theorems 2.1 and 5.4.
On the other hand, if $\chi_1=\no^4$ and  $\chi_2=\nt^4$,
then \eqref{eq:6.3} holds by Theorems 2.1 and 5.5.
Thus the Mellin transforms of the left
and right sides of \eqref{eq:1.4} are the same for all characters.
By taking inverse Mellin transforms, we see that \eqref{eq:1.4}
holds for $jk \ne 0$, which completes the proof of \eqref{eq:1.4}.


\end{document}